\documentclass{amsart}
\usepackage{amsmath,amsfonts,amssymb,stmaryrd}
\usepackage[all]{xy}
\usepackage{graphicx,psfrag,subfigure}

\newtheorem{thm}{Theorem}[section]
\newtheorem{rk}{Remark}
\newtheorem{prop}{Proposition}[section]

\newtheorem{lema}{Lemma}[subsection]

\newcommand{\R}{{\mathbb{R}}}

\newcommand{\T}{{\mathbb{T}}}
\newcommand{\N}{{\mathbb{N}}}

\begin{document}

%\title{Mapas robustamente transitivos con puntos criticos.}

%\title[Un ejemplo de un mapa $C^{2}$ robustamente transitivo que no es $C^{1}$.]{Un ejemplo de un mapa $C^{2}$ robustamente
%transitivo que no es $C^{1}$.}
\title[An example of a map which is $C^2$ robustly transitive but not $C^1$ (robustly transitive)]{An example of a map which is $C^2$ robustly transitive but not $C^1$ (robustly transitive)}

\author[J. Iglesias]{J. Iglesias}
\address{Universidad de La Rep\'ublica. Facultad de Ingenieria. IMERL.
Julio Herrera y Reissig 565. C.P. 11300. Montevideo, Uruguay}
\email{jorgei@fing.edu.uy }

%\author[C. Lizana]{C. Lizana}
%\address{Departamento de Matem\'aticas. Facultad de Ciencias.
%Universidad de los Andes. La Hechicera-M\'erida 5101.
%Venezuela} \email{clizana@ula.ve}

\author[A. Portela]{A. Portela}
\address{Universidad de La Rep\'ublica. Facultad de Ingenieria. IMERL.
Julio Herrera y Reissig 565. C.P. 11300. Montevideo, Uruguay }
\email{aldo@fing.edu.uy }

\date{\today}

\maketitle
\begin{abstract}

The aim of this work is to exhibit an example of an endomorphism of $\T^{2}$ which is $C^2$-robustly transitive
but not $C^1$-robustly
transitive.

\end{abstract}

\section{Introduction }
Robustly transitive maps can be divided into three families: those which are diffeomorphisms, those which are
locally invertible
endomorphisms and
those which are endomorphisms with critical points. Robustly transitive diffeomorphisms have been studied by
many authors
(see \cite{bdp}, \cite{m1}, \cite{m2}),\cite{sh}).
Although many important theorems are known, several problems remain open. Locally invertible endomorphisms have
been less
studied. In \cite{lp}, the authors construct examples and they give necessary conditions that guarantee robust
transitivity
for this class of maps.
In \cite{hg}, examples of this kind of maps in the torus are also constructed, in several homotopy classes and
which are not hyperbolic,
which is something that cannot happen for diffeomorphisms of two-dimensional manifolds, since any robustly
transitive diffeomorphism on a surface
must be hyperbolic. The last family of these maps is the one that has been studied the less, to the point
that only examples are known.
In \cite{br}
the authors constructed the first known example of a robustly transitive map with persisting
critical points, and in
\cite{ilp}
examples of maps of this kind in the torus have been constructed in different homotopy classes.
These examples share the property of having unstable cones. Furthermore, if there is a curve such that
its
tangent vector at each point is contained in the correspondient unstable cone, then its length is expanded by
an uniform constant
greater than one. It is conjectured that (in dimension two) every example must have this property, and that
there are no robustly
transitive endomorphisms on the sphere having persistence of critical points.
Up to now no relevant theorems for this family of maps are known.

Let $f$ be an endomorphism in $\T^2$  and
$f_{*}$ the map induced by $f$ in the fundamental group of $\T^2$.
The map $f_{*}$ can be represented by a square matrix of size two by two with integer coefficients.

The examples constructed in  \cite{ilp} are in the following homotopy classes: expanding
(both eigenvalues of $f_{*}$ have modulus greater than one), non-hyperbolic ($f_{*}$ has an eigenvalue
of modulus one) and hyperbolic (one of the eigenvalues of $f_{*}$ has modulus greater than one and the
other one less than one).
Unfortunately in the expanding case there is a mistake. Trying to correct this mistake, we succeeded in
constructing an example
which is $C^2$- but not $C^1$-robustly transitive.
This happens because for a map $f$ (that we will construct) with critical point, it is possible to find a $C^{1}$ perturbed $g$ of $f$ and an open set $V$ such that $int (g(V))=\emptyset$. But this is something that cannot happen for a $C^{2}$ perturbation of $f$.
We will now give an overview of its construction.
 The starting point of this construction can be any matrix with integer eigenvalues
 $\lambda$ and
 $\mu$, with $1<|\mu|  < |\lambda |$.
 In order to simplify the computations, we will start with the matrix
$A=\left(\begin{array}{cc}
  8 & 0  \\
  0 & 2  \\
\end{array}%
\right)$.
We consider the points
$(\frac{1}{16},\frac{1}{4})$, $(\frac{1}{2},\frac{1}{2})$ and
$( 0,0)$ ($A(\frac{1}{16},\frac{1}{4})= (\frac{1}{2},\frac{1}{2})$  and
$A(\frac{1}{2},\frac{1}{2})= (0,0)    )$.
Two perturbations are made in disjoint neighborhoods of the points
$(\frac{1}{16},\frac{1}{4})$ y $(\frac{1}{2},\frac{1}{2})$.
Outside their supports, the map still coincides with the matrix $A$.

The $C^0$ perturbation which is performed in a neighborhood of the point
$(\frac{1}{16},\frac{1}{4})$
yields a map $f$ with persistent critical points. (We will call $S_f$ the set of critical points of $f$.)
Furthermore, $S_f$ is a one-dimensional manifold.

The $C^1$-perturbation which is performed in a neighborhood of the point
$(\frac{1}{2},\frac{1}{2})$
yields a map $f$ for which there exists $\rho>0$ such that
$f(S_f)\cap B((\frac{1}{2},\frac{1}{2}),\rho)=\{(x,\frac{1}{2}): \ x\in (\frac{1}{2}-\rho,\frac{1}{2}+\rho)\}$.
It is then possible to make a $C^1$-perturbation of $f$ that gives a map $g$ for which an open set $V$
satisfies
$g^{2}(V)\subset\{(x,0): \ x\in[0,1]\}$.
Since the set  $\{(x,0): \ x\in[0,1]\}$ is $g$-invariant,
this implies that $g$ is not transitive.
This perturbation can be made $C^2$, and we will show in the last section that the map $f$ thus obtained is
in fact $C^2$-robustly transitive.
\\

The paper is organized as follows. In Section~\ref{critical} we present
some results regarding maps with critical points in
 $\T^{2}$.
In Section \ref{fyh}
we construct the map $h$ which is $C^2$ but not $C^1$ robustly transitive, and in the last section we
prove that $h$ is $C^2$ robustly transitive.

\section{Critical point in $\T^{2}$}\label{critical}

We recall the following facts borrowed from transversality theory:

\begin{enumerate}
\item
 If $M$ and $N$ are manifolds without boundary and $V$ is a submanifold of $N$,
then $f:M\to N$ is transverse to $V$ if $Df_x(T_xM)+T_{f(x)}(V)=T_{f(x)}N$ when $f(x)\in V$.
\item
If $f$ is transverse to $V$ then $f^{-1}(V)$ is a submanifold of $M$ of codimension equal to the
codimension of $V$ in $N$.
\item
The transversality theorem: under the hypothesis above, and if $V$ is closed:
$$
\{f\in C^1(M,N)\ :\ f \mbox{   is transverse to   }V\}
$$
is open and dense in the strong topology.
%\item
%The parametrized transversality theorem:\\
%If $\Lambda$ is another manifold and $\psi:M\times\Lambda\to N$ is transverse to $V$, then
%$\psi_\lambda$ is transverse to $V$ for almost every $\lambda\in \Lambda$.
\end{enumerate}

Let
$\mathcal{M}_{2\times 2}$
the set of square matrixes of dimension two with real coefficients. Let us recall that
$\mathcal{R}_1=\{A\in \mathcal{M}_{2\times 2}: \ A \mbox{ has  rank } 1\}$ is a closed submanifold of $\mathcal{M}_{2\times 2}$ of
codimension one.

\begin{thm}\label{sf}
There exists a residual subset ${\mathcal S}$ of $C^2(\T^{2})$ such that, for every $f\in {\mathcal S}$
the set $S_f$ is a submanifold of $\T^{2}$ of codimension $1$.
\end{thm}

Idea of the proof:
Consider the operator $D$ which assigns to each $f\in C^2$ its derivative $Df:\T^{2}\to \mathcal{M}_{2\times 2}$.
Define
$${\mathcal S}=\{f\in C^2(\T^{2})\ :\ Df \mbox{  is transverse to  } \mathcal{R}_1\}$$
It is not difficult to prove that the set
${\mathcal S}$ is open and dense.
Therefore, using the transversality properties that have been stated above, we have that,
$S_f=Df^{-1}( \mathcal{R}_1)$
is a one-dimensional manifold for every
$f\in {\mathcal S}$.$\Box$

Let $f\in C^{r}(\T^{2})$, $r\geq 1$. %with $M$  dimension two.
We say that
$(x_0,y_0)\in S_{f}$ is a critical point of \textbf{fold type} if there exist neighborhoods $U$
and $V,$ of $(x_0,y_0)$ and $f(x_0,y_0)$ respectively, and local diffeomorphisms
$\psi_1:\mathbb{R}^{2}\to U$ and $\psi_2:V\to \mathbb{R}^{2}$ such that
$\psi_2\circ f\circ \psi_1 (x,y)=(x,y^{2})$, for every $(x,y)\in U$.

The proof of the following proposition can be read in
\cite{gg}
\begin{prop}\label{fold}
   Let $f\in C^{2}(\T^{2})$, with  $f\in {\mathcal S}$ and $x\in S_f$. If $T_xS_f$ is transverse to
   $Ker(Df_x)$ then $x$ is a critical point of fold type.
\end{prop}
\begin{rk}
It is clear that if
$f\in {\mathcal S}$ and  $x\in S_f$ is such that  $T_xS_f$ is transversal to $Ker(Df_x)$,
then it is possible to find a neighborhood $V_x$ of $x$
and a $C^2$ neighborhood
$\mathcal{U}_f$
of $f$ in such a way that if
$g\in \mathcal{U}_f$ and $y\in S_g\cap V_x$
then
$T_yS_g$ is transverse to $Ker(Dg_y)$.
This implies that every critical point of $g$ belonging to $V_x$ is a critical point of fold type $\forall g\in \mathcal{U}_f$.

\end{rk}

\section{Construction of $f_{\theta\delta}$ and $h$}\label{fyh}

\subsection{Construction of $f_{\theta\delta}$.}\label{f}

The following construction can be carried out for a matrix $A$ whose eigenvalues are different, integers and greater than one.
To make our example easier, we will stick to working with the matrix

$$A=\left(\begin{array}{cc}
  8 & 0  \\
  0 & 2  \\
\end{array}%
\right).$$

Choose $(\frac{1}{16},\frac{1}{4})\in\T^{2}$, note that $A(\frac{1}{16},\frac{1}{4})= (\frac{1}{2},\frac{1}{2})$
and $A^{2}(\frac{1}{16},\frac{1}{4})= (0,0)$. Fix $r>0$  such that

\begin{itemize}\label{p1}
 \item $A(B((\frac{1}{16}, \frac{1}{4}),r))\cap \overline{B((\frac{1}{16}, \frac{1}{4}),r)}=\emptyset $,
 \item $ A^{-1}(B((\frac{1}{2}, \frac{1}{2}),r))\cap \{ (x,0): \ x\in [0,1]          \}=\emptyset$ and
 \item $B((\frac{1}{16}, \frac{1}{4}  ),r)\cap B((\frac{1}{2}, \frac{1}{2}  ),r)=B((\frac{1}{2}, \frac{1}{2}  ),r)\cap B((0, 0  ),r)=$\\
 $B((0,0  ),r)\cap B((\frac{1}{16}, \frac{1}{4}  ),)  = \emptyset.$
 \end{itemize}
This choice of $r$ will play an important role later.

 %
% \begin{equation}\label{p1}
%
%\end{equation}
%\begin{equation}\label{p2}
%
%%\item $\overline{B((x_1,y_1),3r)}\cap \overline{B((x_0,y_0),3r)}=\emptyset$.
%\end{equation}
% \begin{equation}\label{p3}
%
%\end{equation}

Given $\theta >0$ with $2\theta <r$,  consider
$\psi :\R\to\R$ such that $\psi$ is $C^{\infty}$, $x_0=\frac{1}{16}$ is the unique critical point, $\psi ^{''}(\frac{1}{16})=0$,
$\psi (\frac{1}{16})=4$ and $\psi (x)=0$ for $x$ in the complement of the intervals
$(\frac{1}{16}-\theta ,\frac{1}{16}+\theta  )$, see figure \ref{figura1} (a) .
Given $\delta >0$ with $\delta < 2\theta <r$,  consider  $\varphi:\R\to\R$ be such that
\begin{itemize}
\item $\varphi' $ is as in figure \ref{figura1} (b), $\varphi' (\frac{1}{4})=1/2$, $\varphi (\frac{1}{4})=0$, $\varphi^{''}(\frac{1}{4})\neq 0$ , $\varphi' (\frac{1}{4} +\frac{\delta}{8}  )=1$ and
\item $\varphi (y)=0$ for
$y\notin [\frac{1}{4}-\frac{\delta}{4},\frac{1}{4}+\frac{3\delta}{4} ]$.
\end{itemize}
Note that $max \{ |\varphi (y)|: \ y\in \R \}\leq \delta.$\\ Then define $f_{\theta\delta}:\T^{2}\to \T^{2}$  by
$$ \huge{f_{\theta\delta}(x,y)=(8 x,2 y-\psi (x)\varphi (y)).}$$

\begin{figure}[ht]
\psfrag{1}{\tiny{$1$}}
\psfrag{dd}{\tiny{$\frac{1}{4}+\delta/8$}}
\psfrag{d}{\tiny{$\frac{1}{4}+\delta$}}
\psfrag{c}{\tiny{$2\mu$}}
\psfrag{a}{\tiny{$\frac{1}{16}-\theta$}}
\psfrag{aa}{\tiny{$\frac{1}{16}+\theta$}}
\psfrag{q}{$\psi$}
\psfrag{p}{$\varphi'$}

\psfrag{-1}{\tiny{$-1$}}
\psfrag{12}{\tiny{$\frac{1}{2}$}}

\psfrag{14}{\tiny{$\frac{1}{4}$}}

\psfrag{11}{$\frac{1}{16}$}

\psfrag{4}{\tiny{$4$}}

\begin{center}
\subfigure[]{\includegraphics[scale=0.21]{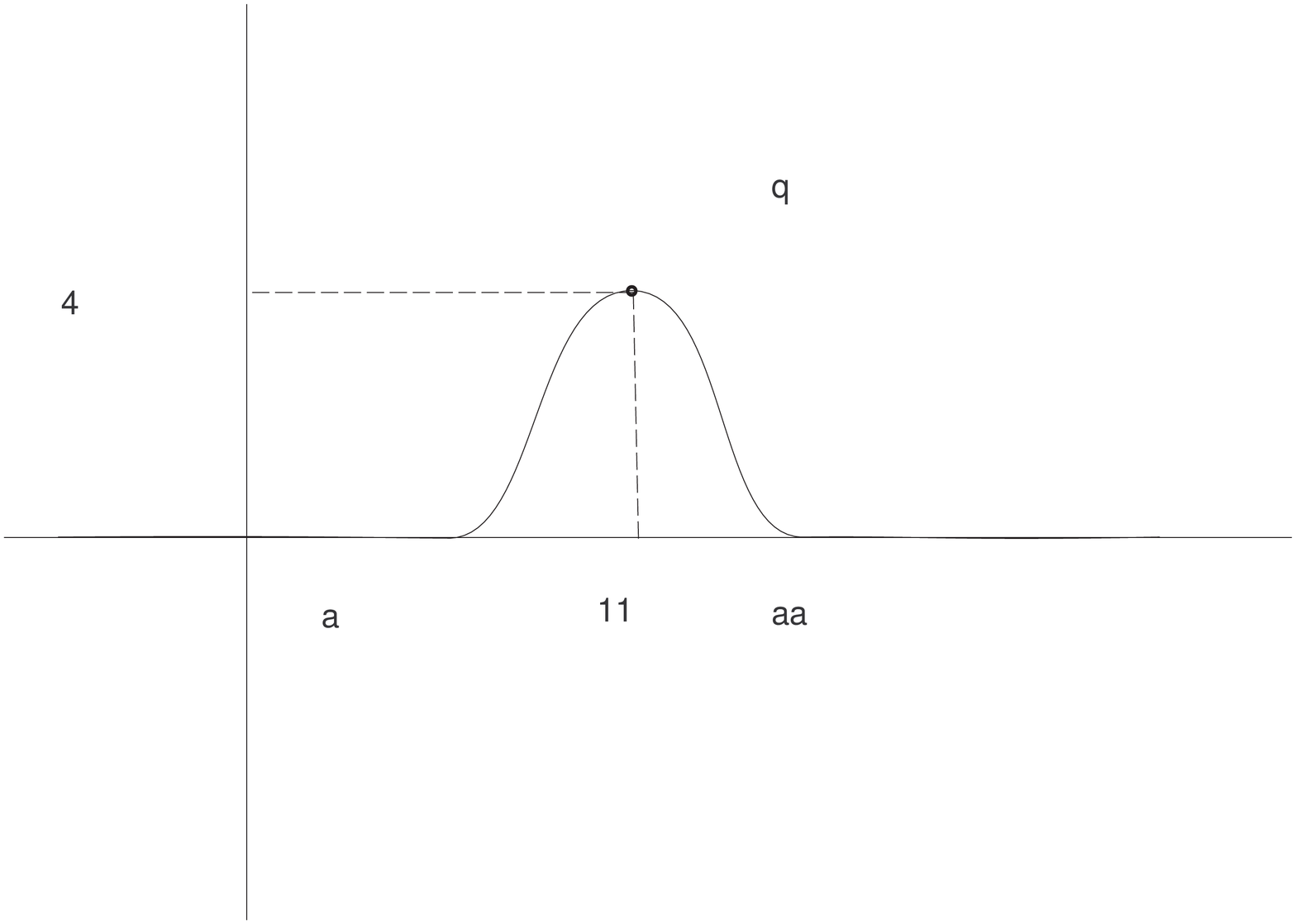}}
\subfigure[]{\includegraphics[scale=0.21]{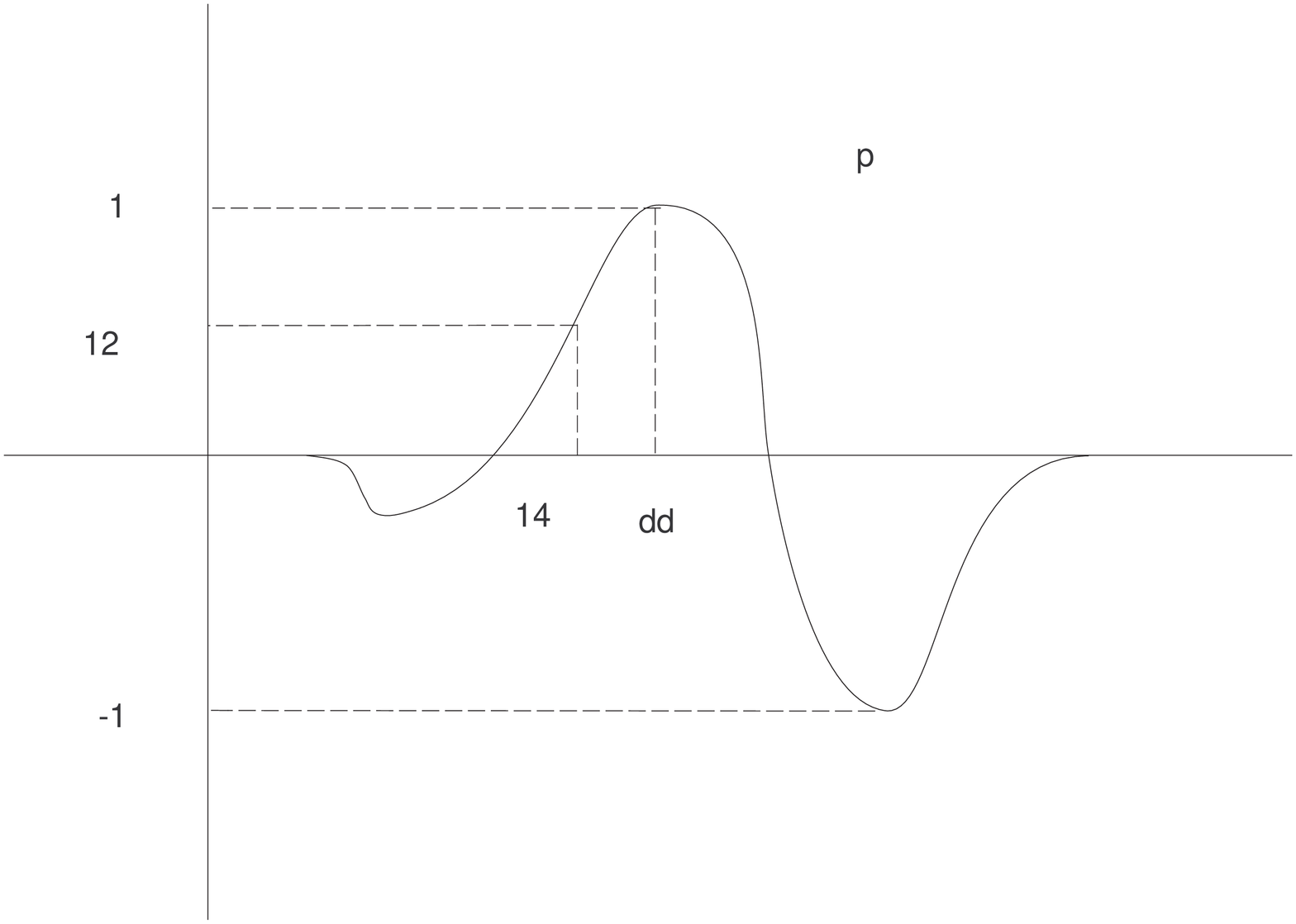}}
\caption{Graph of $\psi$ and $\varphi'$}\label{figura1}
\end{center}
\end{figure}

 For simplicity, sometimes we denote $f_{\theta \delta }=f$, although $f$ depends on the parameters $\theta$ and $\delta$.
Note that $(\frac{1}{16},\frac{1}{4})$ is a critical point and $f(\frac{1}{16},\frac{1}{4})=(\frac{1}{2},\frac{1}{2})$.

\begin{rk}\label{rk1}
   The following properties are very useful for our purpose. Since they are not hard to prove
   we will leave them for the reader to verify.
   \begin{enumerate}
    \item[a)] $f\mid_{\T^{2}\setminus B((\frac{1}{16},\frac{1}{4}),r)} =A_{\T^{2}\setminus B((\frac{1}{16},\frac{1}{4}),r)}.$
    \item[b)] From the definition of $r$, it follows that $f(B((\frac{1}{16},\frac{1}{4}),r))\cap B((\frac{1}{16},\frac{1}{4}),r)=\emptyset$.
    \item[c)] The critical set of $f$ is $S_f=\{(x,y): 2-\psi (x)\varphi' (y)=0 \}$.
    \item[d)]  For $x_0=\frac{1}{16}$ and $y_0=\frac{1}{4}+\delta /8$ we have that $det (Df_{(x_0,y_0)})=-16 $
        and for $(x,y)\in \T^{2}\setminus B((\frac{1}{16},\frac{1}{4}),r)$ we have that  $det (Df_{(x,y)})=16$.
        Then there exists a $C^{1}$-neighborhood  $\mathcal{U}_f$ of $f$ such that
        $S_g\neq\emptyset$ for all $g\in\mathcal{U}_f$.
    \item[e)] $f$ goes to $A$ in the $C^{0}$ topology, when  $\theta$ and $\delta$ go to zero.
    \item[f)]  $f$ is homotopic to $A$.
   \item[g)] $Df$ is transverse to $\mathcal{R}_1$ (here we use that $\varphi^{''}(\frac{1}{4}))\neq 0$ ) , then (by proof of Theorem \ref{sf}) $S_f$ is a manifold of dimension one.

    \end{enumerate}
\end{rk}

Given a positive $a\in \R$  and  $p\in\T^{2}$, we
consider $\mathcal{C}^u_a(p)\subset T_p(\T^{2})$    \emph{the family of unstable cones} defined by
$\mathcal{C}_a^u (p)=\{(v_1,v_2)\in T_p(\T^{2}) : \  |v_2|/|v_1| < a \}$.

The following lemma shows that it is possible to construct a family of unstable cones which is invariant
for the map $f$.

\begin{lema}[Existence of unstable cones for $f$]\label{conosinestables}

Given $\theta >0$, $a>0$, there exist $a_0>0$
and $\delta_0  >0$ with $0<a_0<a$ such that for all $\delta < \delta_0$ if
$f=f_{\theta ,\delta },$ then the following properties hold:
\begin{enumerate}
\item[$(i)$] If $\mathcal{C}_{a_{0}}^u(p)\!\!=\!\!\{(v_1,v_2):   |v_2|/|v_1| <a_0 \},$
        then $\overline{Df_p(\mathcal{C}_{a_{0}}^u(p))}\setminus \{(0,0)\}\subset \mathcal{C}_{{a_{0}}}^u(f(p))$,
        for all  $p\in \T^{2}$;
\item[$(ii)$] If $v\in \mathcal{C}_{a_{0}}^u(p),$ then $|Df_p(v)|\geq 4|v|;$ and
\item[$(iii)$] If $\gamma$ is a small curve such that $\gamma'(t)\subset \mathcal{C}^u_{a_{0}}(\gamma(t)),$
         then $\mathrm{diam}(f(\gamma))\geq 4 \mathrm{diam}(\gamma )$.
\end{enumerate}
\end{lema}

\begin{proof}
Proof of $(i)$. Given $p=(x,y)\in\T^2$ and $a>0,$ pick $a_0$ such that $0<a_0<a.$  Let $v=(v_1,v_2)\in \mathcal{C}_{a_0}^u(p)$. Since

 $$\begin{array}{ll}
                         Df_{(x,y)}(v_1,v_2)&= \left(
                                                 \begin{array}{cc}
                                                 8& 0  \\
                                                 -\psi'(x)\varphi (y) & 2-\psi (x)\varphi' (y)  \\
                                                 \end{array}
                                               \right)
                                                 \left(\begin{array}{c}
                                                     v_1  \\
                                                      v_2  \\
                                                  \end{array}%
                                                   \right) \\ \\
                          & = ( 8 v_1, -\psi'(x)\varphi(y)v_1+ (2-\psi (x)\varphi' (y))v_2),
                       \end{array}
                       $$
then

 \begin{equation}\label{eq1}
 \frac{|-\psi'(x)\varphi (y)v_1+ (2-\psi (x)\varphi' (y))v_2 |}{|8 v_1|}\! \leq\!
 \left|\frac{-\psi'(x)\varphi (y)}{8}\right|\! +\! \left|\frac{2-\psi (x)\varphi' (y)}{8}\right|  \left|\frac{v_2}{v_1}\right|.
\end{equation}

 Let $M=\max\{| \psi'|\}$. Note that $\max\{| \varphi|\}\leq \delta$
 and $|2-\psi (x)\varphi' (y)|\leq |6|.$ Hence, from inequality (\ref{eq1}) follows that
$$     \frac{|-\psi^{'}(x)\varphi (y)v_1+ (2-\psi (x)\varphi' (y))v_2 |}{|8 v_1|} \leq
\frac{M\delta}{8}+\frac{6}{8} \left|\frac{v_2}{v_1}\right|   %\stackrel{\mbox{\scriptsize
%\begin{tabular}{c}
%$\left|\frac{v_2}{v_1}\right|<a_0 $
%\end{tabular}}}
<  \frac{M\delta}{8}+\frac{6 a_0}{8}.$$
Taking $\delta _0$ small enough we obtain that
$\frac{M\delta}{|8|}+\frac{|6 |a_0}{|8|}<a_0$ for all $\delta <\delta_0$, which finishes the proof of (i).

\noindent Proof of $(ii)$.  Note that
$$\begin{array}{ll}
    \left (\dfrac{|Df_{(x,y)}(v_1,v_2)|}{|4(v_1,v_2)|} \right )^{2} &= \dfrac{ (8 v_1)^{2} +(-\psi'(x)\varphi (y)v_1+ (2-\psi (x)\varphi' (y))v_2)^{2}}{ (4)^{2}(v_1^{2}+v_2^{2}) } \\ \\
     & = \dfrac{ 8^{2} +(-\psi'(x)\varphi (y)+ (2-\psi (x)\varphi' (y))\frac{v_2}{v_1})^{2}}{ (4)^{2}\left(1+\left (\frac{v_2}{v_1}\right)^{2}\right) }.
  \end{array}
 $$

Taking $\delta_0$  and $a_0$ small enough, we get that $-\psi'(x)\varphi (y)$ is arbitrarily close to zero,
 $(2-\psi (x)\varphi' (y))\frac{v_2}{v_1}$ is also close to
zero  and $\left(1+\left(\frac{v_2}{v_1}\right)^{2}\right)$ is close to one.
Then %as $8^{'}<|8|$,
there exist $a_0$ and $\delta_0$, as close to zero as necessary, such that
$$  \left (\frac{|Df_{(x,y)}(v_1,v_2)|}{|4(v_1,v_2)|} \right )^{2}>1,$$
and the thesis follows.

\noindent Proof of $(iii)$. It follows from the previous items.
\end{proof}

Note that %if $f$ satisfies the previous lemma then
the properties $(i)$, $(ii)$ and $(iii)$ in the previous lemma are robust.
In concrete we have the following result.

\begin{lema}\label{clly1}
For every $f$ satisfying %thesis of previous
 Lemma \ref{conosinestables}, there exists $\mathcal{U}_f$ a $C^{1}$-neighborhood
 of $f$ such that for every $g\in\mathcal{U}_f$
 the properties $(i)$, $(ii)$ and $(iii)$ of Lemma \ref{conosinestables} hold.
\end{lema}

\subsection{Construction of $h$.}\label{h}
Starting from the map $f$ we will construct another map $h$
having unstable cones and for which there exists a $\rho>0$ such that
$h(S_h)\cap B((\frac{1}{2},\frac{1}{2}),\rho)=\{(x,\frac{1}{2}): \ x\in(\frac{1}{2}-\rho,\frac{1}{2}+\rho)\}$.

The following lemma shows that if $(t, \gamma (t))$ is a small curve which is close to the point  $(\frac{1}{2},\frac{1}{2})$ with $\gamma^{'}$ and  $\gamma^{''}$
close to zero, then it is possible to construct a $C^{2}$ perturbation  $F$ of the identity in such a way that
$F(t,\gamma (t))=(t,\frac{1}{2})$
for values of $t$ close enough to $\frac{1}{2}$.

The proof of the last theorem is conceptually simple but cumbersome in terms of calculations

\begin{lema}\label{l11}
 Given $\varepsilon >0$ and $r>0$ there exist $a >0$, $a^{'}>0$, $a^{''}>0$ and $b>0$ such that if $\gamma :[\frac{1}{2}-b,\frac{1}{2}+b]\to [0,1]$ is a
 function of class $C^{2}$ satisfying
 \begin{itemize}
 \item $max\{ |\gamma (x)-\frac{1}{2}|: \ x\in [\frac{1}{2}-b,\frac{1}{2}+b] \}\leq a$,
  \item $max\{ |\gamma ^{'}(x)|: \ x\in [\frac{1}{2}-b,\frac{1}{2}+b] \}\leq a^{'}$ and
  \item $max\{ |\gamma ^{''}(x)|: \ x\in [\frac{1}{2}-b,\frac{1}{2}+b] \}\leq a^{''}$,

 \end{itemize}

 then there exists a diffeomorphisms $F:\T^{2}\to \T^{2}$ such that  $d_{C^{2}}(F,Id)<\varepsilon $, $F|_{B((\frac{1}{2},\frac{1}{2}), r))^{c}}=Id$ and $F(x,\gamma (x))=(x,\frac{1}{2})$ for all $x\in [\frac{1}{2}-b,\frac{1}{2}+b]$.
\end{lema}
\begin{proof}

Given $\varepsilon >0 $ and $r>0$ let $\delta_0=r/3$ and let $g:[0,1]\to [0,1]$ be as in Figure \ref{figura3} (b).
Let $M^{'}_{\delta_{0}}=max\{ |g^{'}(x)|: \ x\in [0,1] \}$, $M^{''}_{\delta_{0}}=max\{ |g^{''}(x)|: \ x\in [0,1] \}$  and  $2b<r$.\\ Let
  $a^{''}>0$ be such that
$a^{''}<\varepsilon /3$.\\ Consider $a^{'} >0$ such that:

$$\bullet a^{'}<a^{''}b/2 \ \ \ \  \bullet a^{'}<\varepsilon /2 \ \ \ \  \bullet a^{'} M^{'}_{\delta_{0}}<\varepsilon /2  \ \ \ \  \bullet 2(a^{'})^{2} M^{'}_{\delta_{0}}<\varepsilon .$$

And take $a>0$ such that

$$  \bullet a<\varepsilon /2\ \ \ \ \bullet a<a^{'}b/2  \ \ \ \ \bullet a a^{'}M^{'}_{\delta_{0}}<\varepsilon /4 \ \ \ \  \bullet aM^{'}_{\delta_{0}} <\varepsilon /2 .$$

$$  \bullet a a^{'}M^{''}_{\delta_{0}} <\varepsilon /4 \ \ \ \ \bullet a   ( M^{''}_{\delta_{0}} (a^{'})^{2} +M^{'}_{\delta_{0}} a^{''}     )<  \varepsilon /4  \ \ \ \ \bullet a M^{''}_{\delta_{0}}<\varepsilon /4 \ \ \ \  \bullet a<\delta_0 /2.$$

Then, given $\gamma$ in the hypotheses of the lemma, let $f:[0,1]\to [0,1]$, $f\in C^{2}$,  be such that:

\begin{itemize}
\item $f|_{[\frac{1}{2}-b,\frac{1}{2}+b]}=\gamma$.

\item $f(x)=\frac{1}{2}$ if $x\not\in (\frac{1}{2}-2b, \frac{1}{2}+2b)$.
\item $max\{ |f (x)-\frac{1}{2}|: \ x\in [0,1] \}\leq 2a$.

\item $max\{ |f ^{'}(x)|: \ x\in [0,1] \}\leq a^{'}$ (this is possible because  $a<a^{'}b/2$).
\item $max\{ |f ^{''}(x)|: \ x\in [0,1] \}\leq a^{''}$ (this is possible because  $a^{'}<a^{''}b/2$).
\end{itemize}

 and define $$F(x,y)=( x ,y-(f (x)-\frac{1}{2})g(y-f (x)) ).$$
Note that  $F(x,f(x))=( x ,\frac{1}{2})$ for all $x\in [0,1]$.\\
We begin proving that $d_{C^{0}}(F,Id)<\varepsilon $.
$||F(x,y)-Id(x,y)||=||  (0,-(f(x)-\frac{1}{2})g(y-f(x))       ||$, so as $|g|\leq 1$ and $|f-\frac{1}{2}|\leq 2a<\varepsilon $ this implies
    $d_{C^{0}}(F,Id)<\varepsilon $.\\

$$\left|\frac{\partial}{\partial x}( -(f(x)-\frac{1}{2}) g(y-f (x))\right|= \left|-f^{'} (x)g(y-f (x))   +(f(x)-\frac{1}{2})g^{'}(y-f (x))   f^{'}(x) \right|$$ $$ <a^{'}+2aM^{'}_{\delta_{0}}a^{'}    < \frac{\varepsilon}{2}+\frac{\varepsilon}{2}=\varepsilon .$$

$$\left|\frac{\partial}{\partial y}( -(f(x)-\frac{1}{2}) g(y-f (x))\right|=\left| -(f(x)-\frac{1}{2})g^{'}(y-f (x))\right| <  2aM^{'}_{\delta_{0}}<\varepsilon.$$
This implies
    $d_{C^{1}}(F,Id)<\varepsilon $.\\
The proof that is  $d_{C^{2}}(F,Id)<\varepsilon $ analogous.

We will now prove that  $F|_{B((\frac{1}{2},\frac{1}{2}),r))^{c}}=Id$.
For the sake of convenience we will take  $B((\frac{1}{2},\frac{1}{2}),r)=\{ (x,y): \ |x-\frac{1}{2}|<r,|y-\frac{1}{2}| <r  \}$.

 If $(x,y ) \in( B((\frac{1}{2},\frac{1}{2}),r))^{c}$ then $|x-\frac{1}{2}|>r$ or $|y-\frac{1}{2}|>r$.
 If $|x-\frac{1}{2}|>r $, as $r>2b$ then $f (x)=\frac{1}{2}$, and therefore  $F(x,y)=( x ,y)$.
 If $|y-\frac{1}{2}|>r$, as $r=3\delta_0$ and $|f (x)-\frac{1}{2}|<a<\delta_0$, then $|y-f (x)|>\delta _0$. This in turn
 implies that $g(y-f (x))=0$, and therefore $F(x,y)=( x ,y)$.

\end{proof}

In the sequel, the idea is to apply the preceding lemma to
a $\varepsilon$, an $r$ and a $\gamma$ that we will now define.
We will start by defining $\varepsilon$ and $r$.\\
Let $f_{\theta\delta}$ the map defined in Section \ref{f}.
Lemma \ref{conosinestables} guarantees the existence of  $a_0>0$ for which there are $\theta$ and $\delta$
as small as necessary in such a way that
$\overline{Df_{\theta\delta}(\mathcal{C}_{a_{0}}^u(p))}\setminus \{(0,0)\}\subset \mathcal{C}_{{a_{0}}}^u(f_{\theta\delta}(p))$,
        for all  $p\in \T^{2}$.
By Lemma \ref{clly1} there exists a $C^{1}$-neighborhood  $\mathcal{U}_f{_{_{\theta\delta}}}$ of $f_{\theta\delta}$  such that for all $g\in \mathcal{U}_f{_{_{\theta\delta}}}$  we have that $\overline{Dg_p(\mathcal{C}_{a_{0}}^u(p))}\setminus \{(0,0)\}\subset \mathcal{C}_{{a_{0}}}^u(g(p))$ for all  $p\in \T^{2}$.

        Let $\varepsilon^{'} >0$ be such that if $d_{C^{1}}(g,f_{\theta\delta})<\varepsilon^{'}$ then $g\in \mathcal{U}_f{_{_{\theta\delta}}}$ and let
 $\varepsilon>0$ be such that if  $d_{C^{1}}(F,Id)<\varepsilon$ then $d_{C^{1}}(F\circ f_{\theta\delta},f_{\theta\delta})<\varepsilon^{'}$, so $F\circ f_{\theta\delta}\in  \mathcal{U}_f{_{_{\theta\delta}}}$.

 Let $r>0$ be as defined at the beginning of section \ref{f}. Let us fix $\varepsilon$ and $r$.
 The idea is to apply Lemma
 \ref{l11}, and in order to do this we still have to define the curve $\gamma$. Some preliminary remarks are in order.

Recall that $f_{\theta\delta}(x,y)=f(x,y)=(8x, 2x-\psi (x)\varphi(y)) $ and  $S_f=\{ (x,y): \  2-\psi (x)\varphi ^{'}(y)=0 \}$.
We consider a point $(\frac{1}{16},\frac{1}{4} )$. Recall that $(\frac{1}{16},\frac{1}{4})\in S_f$ and $f_{\theta\delta}(\frac{1}{16},\frac{1}{4})=(\frac{1}{2},\frac{1}{2})$ .

By the implicit function Theorem applied to  $G(x,y)=2-\psi (x)\varphi ^{'}(y)=0$ in the point $(\frac{1}{16},\frac{1}{4})$ (here we using that $\varphi'' (\frac{1}{4})\neq 0$), there exist $\beta >0$ and a function $y:(\frac{1}{16}-\beta, \frac{1}{16}+\beta)\to \R$ such that $G(x,y(x))=0$ for $x\in(\frac{1}{16}-\beta, \frac{1}{16}+\beta)$.
A straightforward computation shows that $y^{'}(1/16)=y^{''}(1/16)=0$ (here we use that $\psi^{'}(1/16)=\psi^{''}(1/16)=0$).

So $f(x,y(x))=(8x, 2y(x)-\psi (x) \varphi (y(x)) )$. By the change of variables  $8x=t$ we obtain $(t, 2y(t/8)-\psi (t/8) \varphi (y(t/8)) )=(t,\gamma_1(t))$. Again a straightforward computation shows that $\gamma^{'}_1(1/2)=\gamma_1^{''}(1/2)=0$ and  $\gamma_1(1/2)= 1/2$.

For $\varepsilon$ a $r$ as above,  let $a,a^{'},a^{''}$ and $b$ as given by the Lemma \ref{l11}. Let $\beta >0$ enough small such that $8\beta <b$ and
\begin{enumerate}
 \item $max\{ |\gamma_1 (t)-\frac{1}{2}|: \ t\in [\frac{1}{2}-8\beta,\frac{1}{2}+8\beta] \}\leq a$,
  \item $max\{ |\gamma_1 ^{'}(t)|: \ t\in [\frac{1}{2}-8\beta,\frac{1}{2}+8\beta] \}\leq a^{'}$ and
  \item $max\{ |\gamma_1 ^{''}(t)|: \ t\in [\frac{1}{2}-8\beta,\frac{1}{2}+8\beta] \}\leq a^{''}$,
\end{enumerate}
in such way that there exists a function $\gamma:(\frac{1}{2}-b,\frac{1}{2}+b )\to \R$ whit $\gamma|_{ (\frac{1}{2}-8\beta,\frac{1}{2}+8\beta) }=\gamma_1$ and the properties (1), (2) and (3) hold.

 Therefore, by Lemma \ref{l11}, there exist $F:\T^{2}\to \T^{2}$ such that $d_{C^{2}}(F,Id)<\varepsilon$, $F|_{B((\frac{1}{2},\frac{1}{2}), r))^{c}}=Id$ such that $F(x,\gamma (x))=(x,\frac{1}{2})$ for all $x\in [\frac{1}{2}-b,\frac{1}{2}+b]$.\\
Then define $h:\T^{2}\to \T^{2}$ such that

 $${\bf{h=F\circ f_{}}}$$

\begin{rk}\label{propiedades}
 $h$ satisfies the following properties

   \begin{enumerate}

  \item $S_h=S_{f}$, because $F$ is a diffeomorphisms.
\item $h( B((\frac{1}{16},\frac{1}{4}),r)  )\cap B((\frac{1}{16},\frac{1}{4}),r)=\emptyset$  (see definition $f$ and $r$).

   \item $\overline{Dh_p(\mathcal{C}_{a_{0}}^u(p))}\setminus \{(0,0)\}\subset \mathcal{C}_{{a_{0}}}^u(h(p))$,
        for all  $p\in \T^{2}$, because $h\in \mathcal{U}_f{_{_{\theta\delta}}}$.

   \item There exists $\rho >0$ such that $h(S_h)\cap  B((\frac{1}{2},\frac{1}{2}), \rho)) =\{(x,\frac{1}{2}): \ x\in (\frac{1}{2}-\rho,\frac{1}{2}+\rho )   \}$.
\item $h(x)=A(x)$ for all  $x\in    \T^{2}\setminus        B((\frac{1}{16},\frac{1}{4}), r))\cup A^{-1} (  B((\frac{1}{2},\frac{1}{2}), r)))$.

   \item $h (\{(x,\frac{1}{2}): \ x\in [\frac{1}{2}-\rho,\frac{1}{2}+\rho ]   \})\subset \{(x,0): \ x\in [0,1 ]   \}$.

\item  $h(  \{(x,0): \ x\in [0,1 ]   \})=A(\{(x,0): \ x\in [0,1 ]   \}) \subset \{(x,0): \ x\in [0,1 ]   \}$.

   \item $h(\frac{1}{16},\frac{1}{4})= (\frac{1}{2},\frac{1}{2})$.

\item Since $D{f_{}}$ is transverse to $\mathcal{R}_1$ and
 $d_{C^{2}}(F,Id)<\varepsilon$ (taking a smaller $\varepsilon$ if necessary)  then $D{h_{}}$ is transverse to $\mathcal{R}_1$.

\item $h$ is expanding in $\T^{2}\setminus B((\frac{1}{16},\frac{1}{4}),2r)$.

   \end{enumerate}
   \end{rk}

Now we need prove that $(\frac{1}{16},\frac{1}{4})$ is a critical point of fold type for $h$.

 If $(x,y)\in S_f$ then  $$Df_{(x,y)}=\left(\begin{array}{cc}
  8 & 0  \\
  -\psi^{'}(x)\varphi(y)  &  0 \\
\end{array}%
\right).$$
Since $Df_{(x,y)}(0,1)=(0,0)$,  $Ker(Df_x)$ is generated by the vector $(0,1)$ for all $(x,y)\in S_f$.

A straightforward computation shows that
$T_{(\frac{1}{16},\frac{1}{4})}S_f=(1,0)$, and (by proposition \ref{fold}) this implies that $(\frac{1}{16},\frac{1}{4})$
is a critical point of fold type for $f$.  As $S_h=S_{f}$, $d_{C^{2}}(F,Id)<\varepsilon$(taking a smaller $\varepsilon$ if necessary)   then $T_{(\frac{1}{16},\frac{1}{4})}S_{h}$ is transverse to $Ker (Dh)$,
 which in turn implies that $(\frac{1}{16},\frac{1}{4})$ is a critical point of fold type for  $h$.

 %As $T_{(\frac{1}{16},y_{0})}S_{f}$ is transverse to $Ker (Df)$,
%
%
%
%and if $x\in S_f$ is such that $T_xS_f$ is transverse to $Ker(Df_x)$
% then, due to Proposition \ref{fold}, $x$ is a critical point of fold type.\\
%
%
%
%We consider a point $y_0$ such that $\varphi ^{'}(y_0)=\frac{1}{2}$. Therefore, $(\frac{1}{16},y_0)\in S_f$, since
%$\psi (\frac{1}{16})=4$.
%On the other hand,
%
%
% since $\psi^{'}(\frac{1}{16})=0$,
%
%$$Df_{(\frac{1}{16},y_0)}=\left(\begin{array}{cc}
%  8 & 0  \\
%  -\psi^{'}(\frac{1}{16})\varphi^{}(y_0)  &  0 \\
%\end{array}%
%\right) = \left(\begin{array}{cc}
%  8 & 0  \\
%   0 & 0  \\
%\end{array}%
%\right).$$
%This implies that $T_{f(\frac{1}{16},y_0)}f(S_f)=(8,0)$, and therefore $f(S_f)$
%is the graph of a function in a neighborhood of the point
%$f(\frac{1}{16},y_0)$. Furthermore,
%$ f(\frac{1}{16},y_0)\to (\frac{1}{2},\frac{1}{2})  \mbox{ when } \theta \mbox{ and } \delta \mbox{ go to zero}. $
%
%
%
%It is then clear that the curve
% $f(S_f)$ can be parametrized, in a neighborhood of the point $f(\frac{1}{16},y_0)$,
% in such a way as to obtain  $\gamma : [\frac{1}{2}-b,\frac{1}{2}+b]\to [0,1]$, $(x,\gamma (x))\in f(S_f)$, for $x\in[\frac{1}{2}-b,\frac{1}{2}+b]$,  which is in the hypothesis of Lemma \ref{l11}.

The map $h$ is the one which is  $C^{2}$ robustly transitive but not $C^{1}$.

\section{\bf{Proof of the fact that $h$ is not $C^{1}$ robustly transitive.}}

Proving that $h$ is not $C^{1}$
robustly transitive is simple.
Due to properties (4),(6) and (7) of
Remark \ref{propiedades},  $h(S_h)$ and $\{(x,\frac{1}{2}): \ x\in [0,1 ]   \}$ coincides locally  and
$h (\{(x,\frac{1}{2}): \ x\in [\frac{1}{2}-\rho,\frac{1}{2}+\rho ]   \})\subset \{(x,0): \ x\in [0,1 ]   \}$.
The following lemma will allow us to find a neighborhood $V$ of the point $(\frac{1}{16},\frac{1}{4})$ and a $C^{1}$-perturbation, $g$, of $h$
in such a way that $g^{n}(V)\subset \{(x,0): \ x\in [0,1 ]   \}$ for $n\geq 2$, which means that it is impossible for $g$ to be transitive.

The key is given by the following lemma:

\begin{lema}\label{perturbado}
  Let $x_0\in S_f$ be a critical point of fold type. Then for every $\varepsilon >0$  and every neighborhood $W_{x_{0}}$ of $x_0$
  there exists  $g$ with $d_{C^{1}}(g,f)<\varepsilon $ and a neighborhood $V_{x_{0}}$ of $x_0$ such that
  $V_{x_{0}}\subset S_g$, $g(V_{x_{0}})=f(S_f\cap V_{x_{0}})$ and $g(x)=f(x)$ for every $x\in W_{x_{0}}^{c}$.
\end{lema}
\begin{proof}
As $x_0$ is a critical point of type fold, without loss of generality we can assume that $x_0=(0,0)$ and $f(x,y)=(x,y^{2})$, for $(x,y)\in U$, with $U$
small enough and $U\subset W_{x_{0}}$.
Given $\varepsilon>0,$ we choose $\delta>0$ such that $4\delta<\varepsilon$, $\delta^2<\varepsilon$ and
$B((0,0),\delta )\subset U$.
We consider a bump function $\varphi:\mathbb{R}\to \mathbb{R}$  of class $C^1$ such as that in
Figure \ref{figura3} (a) with
$|\varphi'|\leq \frac{2}{\delta}$.
Consider $g$ a $C^1$-perturbation of $f$ defined by $g(x,y)=(x,\varphi(y)y^2)$, for $(x,y)\in U.$
%, where $\varphi$ is as in figure \ref{figura3}.
Since it is not hard to prove that the distance between $g$ and $f$ is less than $\varepsilon$
in the $C^1$ topology, we leave the details to the reader.
Moreover, the critical set of $g$ contains the ball centered  at the
origin of radius $\frac{\delta}{2}$ and $g(B((0,0), \frac{\delta}{2}))=f(S_f\cap B((0,0), \frac{\delta}{2}) )$.
\end{proof}

\begin{figure}[ht]
\psfrag{d}{\tiny{$\delta$}}\psfrag{d2}{\tiny{$\delta/2$}}
\psfrag{-d2}{\tiny{$-\delta /2$}}
\psfrag{-d}{\tiny{$-\delta$}}
\psfrag{1}{\tiny{$1$}}
\psfrag{aa}{\tiny{$\theta$}}
\psfrag{q}{\tiny{$\psi$}}
\psfrag{phi}{\tiny{$\varphi$}}
\psfrag{-delta0}{{\tiny{$-\delta _0$}}}\psfrag{delta0}{{\tiny{$\delta _0$}}}
\psfrag{g}{$g$}
\begin{center}
\subfigure[]{\includegraphics[scale=0.15]{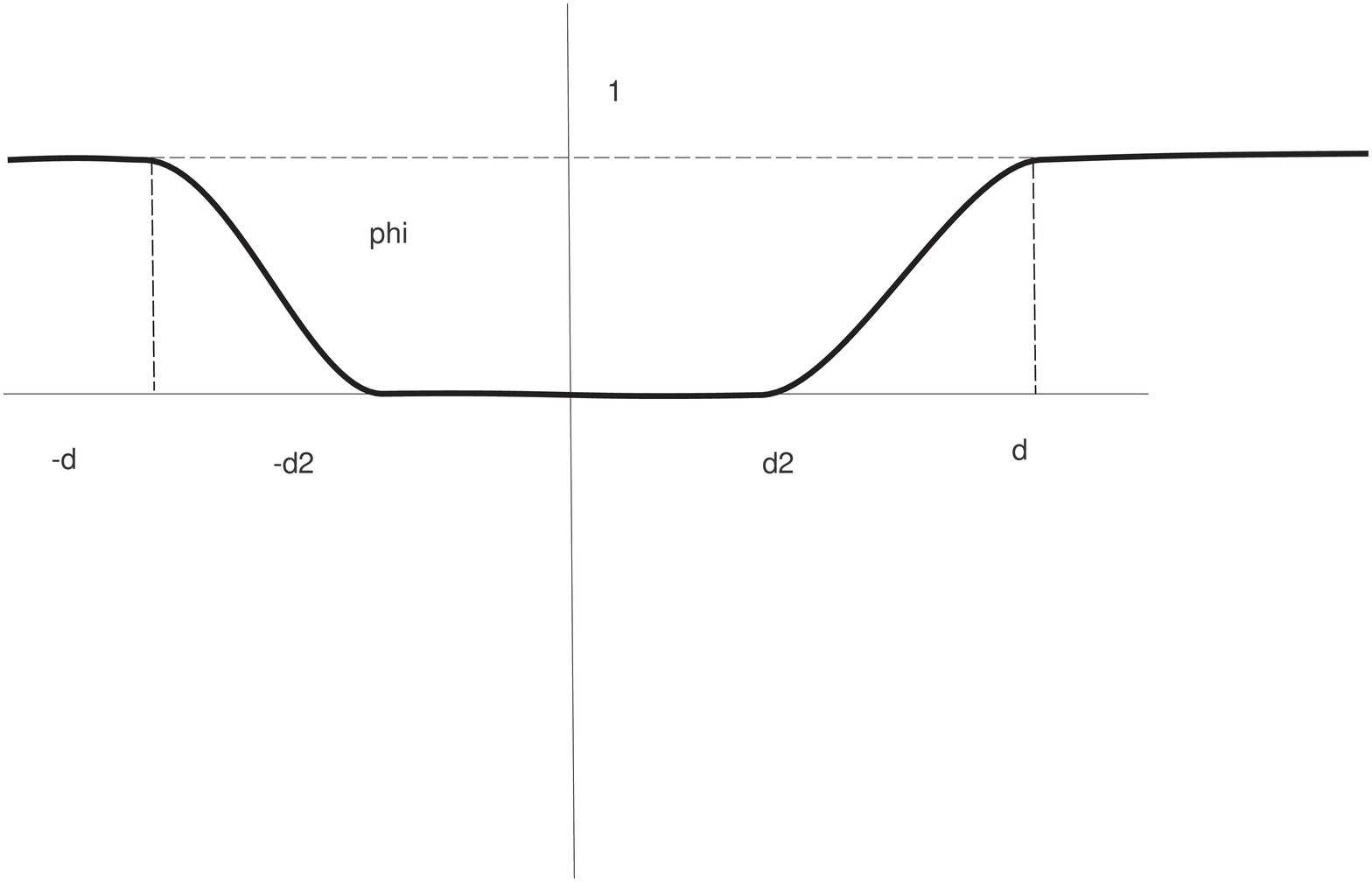}}
\subfigure[]{\includegraphics[scale=0.15]{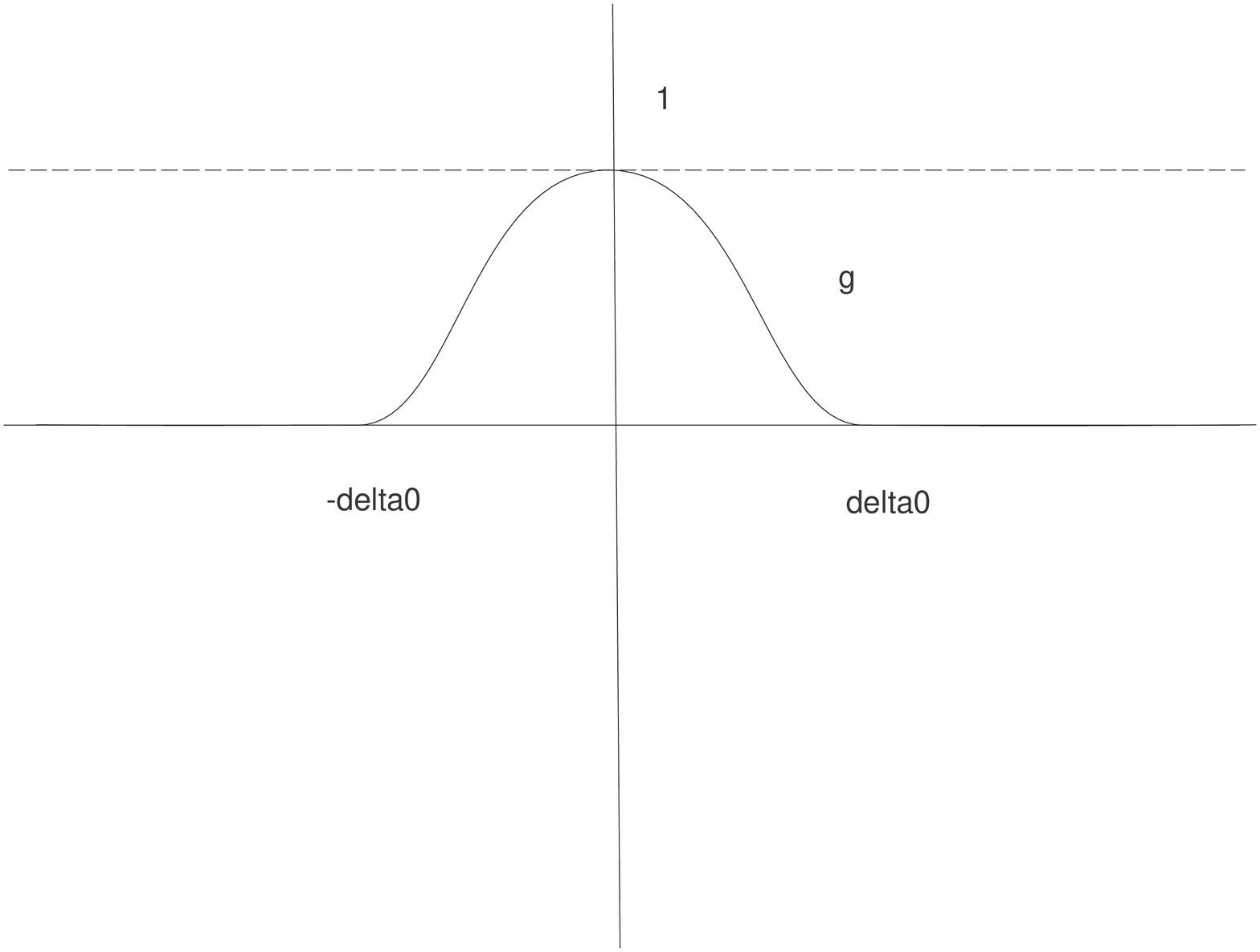}}
\caption{}\label{figura3}
\end{center}
\end{figure}

{\bf{Proof of the fact that h is not robustly transitive:}}\\
As we have previously shown  the point $(\frac{1}{16},\frac{1}{4})$ is a critical point of fold type
$h(\frac{1}{16},\frac{1}{4})=( \frac{1}{2}, \frac{1}{2})$ and there exists $\rho>0$ such that
$h(S_h)\cap  B((\frac{1}{2},\frac{1}{2}), \rho)) =\{(x,\frac{1}{2}): \ x\in (\frac{1}{2}-\rho,\frac{1}{2}+\rho )   \}$.
Then, due to Lemma \ref{perturbado} there exist a neighborhood $V$ of the point
$(\frac{1}{16},\frac{1}{4})$ and $g$, $C^{1}$
close to $h$, such that $g(V)=h(S_h\cap V)\subset \{(x,\frac{1}{2}): \ x\in [0,1 ]   \}$.
As $h (\{(x,\frac{1}{2}): \ x\in (\frac{1}{2}-\rho,\frac{1}{2}+\rho )   \})\subset \{(x,0): \ x\in [0,1 ]   \}$
and $h(  \{(x,0): \ x\in [0,1 ]   \})=A(\{(x,0): \ x\in [0,1 ]   \})$, then
$g^{n}(V)\subset \{(x,0): \ x\in [0,1 ]   \}$ for all $n\geq 2$, which implies that $g$ is not transitive.

\section{\bf{Proof of the fact that $h$ is $C^{2}$ robustly transitive.}}
We recall that the critical set of the function
$f_{}$ is $S_{f_{}}=\{(x,y): 2-\psi (x)\varphi^{'}(y)=0\}$.
We also recall that
$Df$ is transverse to $\mathcal{R}_1$, which implies that $S_{f_{}}$ is a one-dimensional submanifold.

For $y_0=\frac{1}{4}+\frac{\delta}{8}$ ($\varphi ^{'}(y_0)=1$) let $x_0,x_1$, $x_0\neq x_1$ be such that
$(x_0,y_0),(x_1,y_0)\in S_f$. Since $\varphi^{''}(y_0)=0$, we have that
$T_{(x_0,y_0)}S_{f_{}} =T_{(x_1,y_0)}S_{f_{}} =(0,1)$. Since $Ker Df_{}$ is generated by $(0,1)$, is not hard to see that
the points $(x_0,y_0)$ and $(x_1,y_0)$ are the only ones which are not of fold type.
Therefore $f_{}$ satisfies the following properties:
\begin{enumerate}
\item There exists $\rho >0$ such that if $B_0=B((x_0,y_0),\rho )$ and $B_1=B((x_1,y_0),\rho )$ with
a sufficiently small $\rho$, the following holds:
  \begin{itemize}
 \item $B_i\setminus S_{f_{}}$ has two connected components, for $i=0,1.$
\item $S_f\setminus B_0\cup B_1$ are critical points of fold type.
 \item If $\gamma $ is a curve with  $\gamma^{'}(t)\in  \mathcal{C}_{a_{0}}^u (\gamma (t))$ and such that
 $\gamma\cap S_{f_{}}\subset B_0\cup B_1$, then $\gamma$ is transverse to $S_{f}$.
\item If furthermore $diam (\gamma )<10r$ (taking a smaller $r$ if necessary) then $\gamma$ intersects
$B_0\cup B_1$ in at most two points.
\end{itemize}

 \item $f_{}|_{\T^{2}\setminus  B((\frac{1}{16},\frac{1}{4}),r)}$ is expanding, i.e. the eigenvalues
of $Df_{}(p)$ are greater than one for $p\in \T^{2}\setminus B((\frac{1}{16},\frac{1}{4}),r)$.
  \item $f_{}$ satisfies Lemma \ref{conosinestables}.

  \item $Df_{}$ is transverse to $\mathcal{R}_1$.
\end{enumerate}

Let $\mathcal{U}_{f_{}}$ be a $C^{2}$ neighborhood of $f_{}$
that satisfies the properties that have been stated above. Then, since
$h=F\circ f_{}$ and $F$ is as $C^{2}$-close to the identity as necessary, we can take
$h\in \mathcal{U}_{f_{}}$ and $\mathcal{U}_{h}$ a $C^{2}$ neighborhood of $h$ such that
$\mathcal{U}_{h}\subset \mathcal{U}_{f} $.

For the following lemmas we will take $g\in \mathcal{U}_{h}$.
Therefore, if we prove that $g$ is transitive this will mean that $h$ is $C^{2}$ robustly transitive.

We start by making the following remark.

\begin{rk}\label{rk1}

Let $x_0\in S_g$ be a critical point of fold type.
Without loss of generality, we can assume that $x_0=(0,0)$ and $g(x,y)=(x,y^{2})$
for all $(x,y)$ belonging to a small enough neighborhood $U$ of $x_0$.
If $\alpha\subset B(0,r)\subset U $  is a $C^{1}$ curve which is transverse to $S_g$ with $\alpha \cap S_g=\{0\}$,
then $g(\alpha)\setminus g(0)$ is contained in the interior of $g(B(0,r))$. Given $\varepsilon >0$, it is possible to construct a curve $\gamma$ which satisfies the following conditions (see figure \ref{figura4} (b)):
\begin{itemize}
\item $\gamma = g(\alpha )$ in  $B(0,\varepsilon )^{c}$.
\item $\gamma \subset int(g (B(0,r)))$.
\item $\gamma$ is as $C^{1}$-close to $g(\alpha)$ as necessary.

\end{itemize}

\end{rk}

\begin{figure}[ht]
\psfrag{alpha}{$\alpha$}\psfrag{galpha}{$g(\alpha )$}
\psfrag{b1}{$B((0,0),\varepsilon )$}
\psfrag{gamma}{$\gamma$}
\psfrag{z1}{$z_1$}
\psfrag{z2}{$z_2$}
\psfrag{bb}{$B((0,0),r)$}
\psfrag{sf}{$S_f$}
\psfrag{1}{\tiny{$1$}}
\psfrag{aa}{\tiny{$\theta$}}
\psfrag{q}{\tiny{$\psi$}}
\psfrag{phi}{\tiny{$\varphi$}}
\psfrag{-delta0}{{\tiny{$-\delta _0$}}}\psfrag{delta0}{{\tiny{$\delta _0$}}}
\psfrag{g}{$g$}
\begin{center}
\subfigure[]{\includegraphics[scale=0.2]{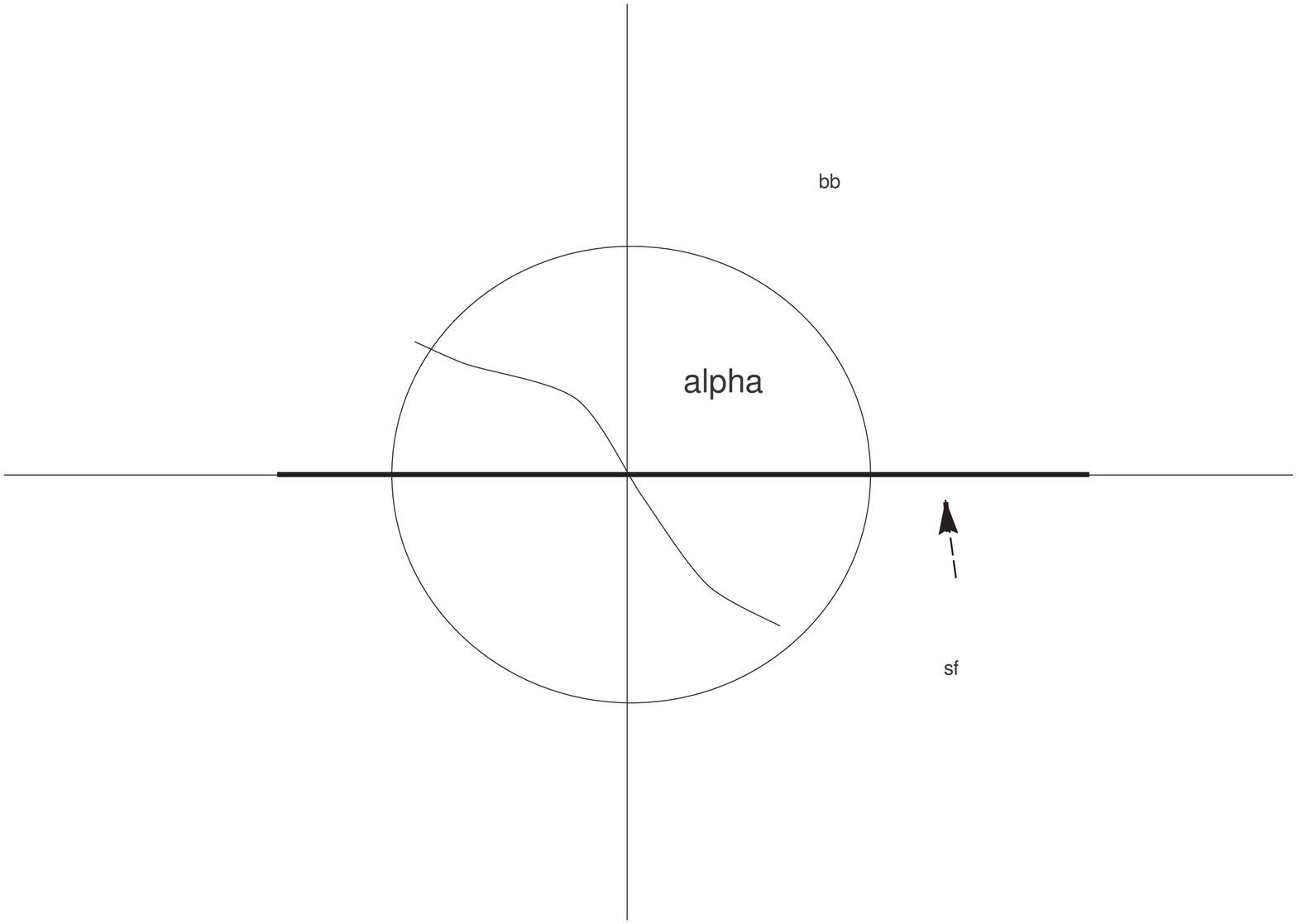}}
\subfigure[]{\includegraphics[scale=0.2]{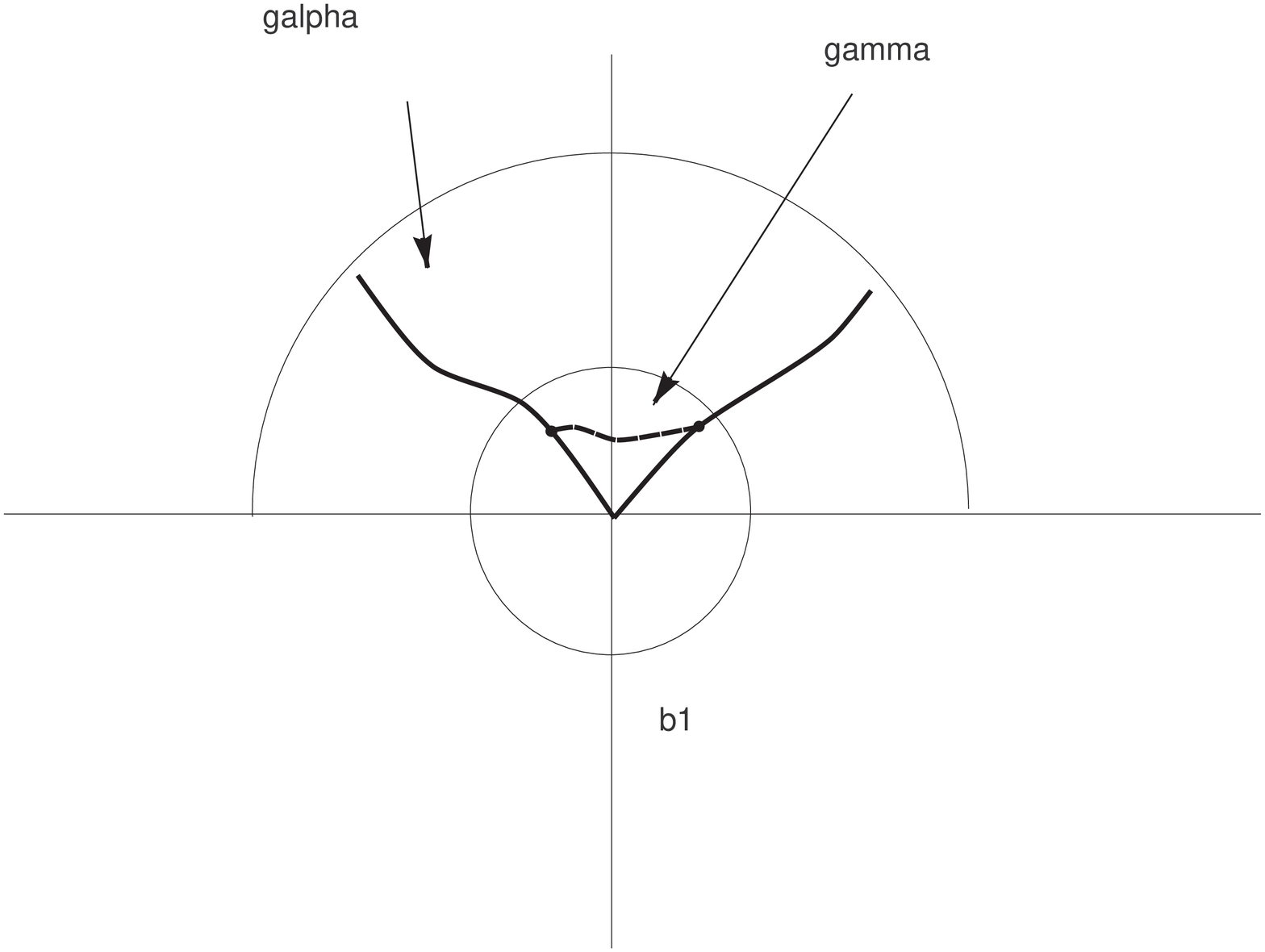}}
\caption{}\label{figura4}
\end{center}
\end{figure}

\begin{lema}\label{lemacurva}
Let $V\subset \T^{2}$ be an open set and $\gamma \subset V$ be a curve such that
$\gamma^{'}(t)\in  \mathcal{C}_{a_{0}}^u (\gamma (t))$,  $\gamma$ is transverse to $S_g$
and $\gamma \cap S_g$ is a single point of fold type. Then there exists a curve $\beta$  with
$\beta^{'}(t)\in  \mathcal{C}_{a_{0}}^u (\beta (t))$ contained in the interior of $g(V)$, transverse to
$S_g$ such that      $\mathrm{diam}(\beta) > 4\mathrm{diam}(\gamma)$.
\end{lema}

\begin{proof}
Since $\gamma$ intersects $S_g$ in a single point $x_0$ fold type,  $g(\gamma)$
is contained in the interior of $g(V)$ except for the point $g(x_0)$ which can be in the boundary of $g(V)$
and furthermore $\mathrm{diam}(g(\gamma)) > 4\mathrm{diam}(\gamma)$ because $\gamma^{'}(t)\in  \mathcal{C}_{a_{0}}^u (\gamma (t))$. Then, due to Remark \ref{rk1}
the curve $g(\gamma)$
can be approximated in a neighborhood of the point $g(x_0)$
by a curve $\beta$ that satisfies the thesis of the lemma.
  \end{proof}

Note that if $\gamma$ is transverse to $S_g$ then $\gamma \cap S_g$
consists of finitely many points. If additionally the intersection points of $\gamma$
with $S_g$ are critical points of fold type, then Lemma \ref{lemacurva} holds.

\begin{lema}\label{lema55}
  Let $V$ be open, $\alpha\subset V$ with  $\alpha'(t)\subset \mathcal{C}^u_{a_0}(\alpha(t))$,
  $\alpha$ transverse to $S_g$ and $diam (\alpha )<10r$. There exists $\beta \subset int(g(V))$,
  with $\beta'(t)\subset \mathcal{C}^u_{a_0}(\beta(t))$, $\beta$ transverse to $S_g$ such that
  $diam (\beta )>\frac{6}{5}diam (\alpha )$.
\end{lema}

As $diam (\alpha )<10r$ then $\alpha\cap S_g$ has at most two points $x_1,x_2$ in $B_0\cup B_1$.
Let $\alpha_1$ be a connected component of $\alpha \setminus \{x_1,x_2\}$ such that
$diam (\alpha_1)\geq\frac{3}{10} diam (\alpha)$.
Note that $\alpha_1$ can only intersect $S_g$ in a finite number of
points of fold type. Then by Lemma \ref{lemacurva},
we can construct a curve $\beta$ with $\beta \subset g(V)$,
$\beta'(t)\subset \mathcal{C}^u_{a_0}(\beta(t))$ and
$diam (\beta )>4 diam(\alpha_1 )=\frac{6}{5}diam (\alpha )$.

\begin{lema}\label{lema66}
  Let $V\subset \T^{2}$ be an open set. Then
there exist $y\in V$ and $n_0\in\N$  such that $g^{n}(y)\in int(g^{n}(V))$ and
$g^{n}(y)\in \T^{2}\setminus B((\frac{1}{16},\frac{1}{4}),2r)$ for all $n\geq n_0$.
\end{lema}

\begin{proof}

\noindent{\bf{Claim :}} If $\alpha$ is a curve such that
$\alpha'(t)\subset \mathcal{C}^u_{a_0}(\alpha(t))$,
$\alpha\subset \T^{2}\setminus B((\frac{1}{16},\frac{1}{4}),2r)$ and
  $diam(\alpha)\geq 2r$,  then there exists $x\in \alpha$ such that
  $g^{n}(x)\in\T^{2}\setminus B((\frac{1}{16},\frac{1}{4}),2r)$ for all $n\geq 0$.

  Proof of claim: we begin by constructing a sequence of curves $\{\alpha_n\}$ such that:
\begin{enumerate}
\item  $g(\alpha_n)\supset \alpha_{n+1}$,$\forall n\in \N$,
\item $\alpha_n\subset \T^{2}\setminus B((\frac{1}{16},\frac{1}{4}),2r)$, $\forall n\in \N$ and
  \item $\alpha'_n(t)\subset \mathcal{C}^u_{a_0}(\alpha _n(t))$.
\end{enumerate}

  Let $\alpha_0=\alpha$. As  $\alpha'(t)\subset \mathcal{C}^u_{a_0}(\alpha(t))$ and  $diam(\alpha)\geq 2r$
  then $diam (g(\alpha ))\geq 8r$. As  $diam (B((\frac{1}{16},\frac{1}{4}),2r))=4r$, there exists
$\alpha_1\subset g (\alpha )$ with $diam (\alpha_1)\geq 2r$ and
$\alpha_1\subset \T^{2}\setminus B((\frac{1}{16},\frac{1}{4}),2r)$.
Proceeding inductively we have the sequence  $\{\alpha_n\}$.\\
For every $\alpha_n,$ let $\gamma_n\subset\alpha_0=\alpha$ be such that
$g^{n}(\gamma_n)=\alpha_n$. Since
$g(\alpha_n)\supset\alpha_{n+1}$ then $\gamma_{n+1}\subset
\gamma_n$. Consider $B=\cap\gamma_n$ (note that $B\neq\emptyset$). If $y\in B,$ then
 $g^{n}(y)\in \T^{2}\setminus B((\frac{1}{16},\frac{1}{4}),2r)$ for all $n\geq 0$, and this proves the Claim.

Given $V$, let $\alpha \subset V$ be such that $\alpha'(t)\subset \mathcal{C}^u_{a_0}(\alpha(t))$ and
$\alpha$ transverse to $S_g$. Then by Lemma \ref{lema55} there exists $\beta_1$ such that:
\begin{itemize}
  \item $\beta_1\subset int(g(V))$,
  \item $diam (\beta_1)\geq   \frac{6}{5}diam (\alpha )$,
 \item $\beta_1$ transverse to $S_g$ and
\item $\beta_1'(t)\subset \mathcal{C}^u_{a_0}(\beta_1(t))$.
\end{itemize}
Proceeding inductively there exist $n_0\in \N$ and $\beta_{n_{0}}$
with the above properties and such that $diam (\beta_{n_{0}})\geq 8r$. Therefore there exist $\widetilde{\beta_{n_{0}}}\subset \beta_{n_{0}}$ such that $diam (\widetilde{\beta_{n_{0}}})\geq 2r$ and $\widetilde{\beta_{n_{0}}} \subset \T^{2}\setminus B((\frac{1}{16},\frac{1}{4}),2r)$.
Then by de claim there exists $x\in \widetilde{\beta_{n_{0}}}$ such that
$g^{n}(x)\in \T^{2}\setminus B((\frac{1}{16},\frac{1}{4}),2r)$.
As $g|_{\T^{2}\setminus B((\frac{1}{16},\frac{1}{4}),2r)}$ is a local diffeomorphism,
we have that $g^{n}(x)\in int( g^{n+n_{0}}(V))$. Taking $y\in V$ such that $g^{n_{0}}(y)=x$, we are done.

\end{proof}

\begin{lema}\label{lema3}
 Given $V$ an open set of $\T^{2}$. There exist $y\in V$ and $n_{_{V}}\in\N$
 such that if $n\geq n_{_{V}}$ then $g^{n}(V) \supset B(g^{n}(y),r)$.
\end{lema}

\begin{proof}
 Given $V$, by Lemma \ref{lema66} there exist $y\in V$ and $n_0\in \N$ such that
 $g^{n}(y)\in int(g^{n}(V))\cap  (\T^{2}\setminus B((\frac{1}{16},\frac{1}{4}),2r))$
 for all $n\geq n_0$. As $g|_{ \T^{2}\setminus B((\frac{1}{16},\frac{1}{4}),r)}$ is expanding and
 $dist(\partial B((\frac{1}{16},\frac{1}{4}),2r), \partial B((\frac{1}{16},\frac{1}{4}),r))=r$
 then there exists $n_{_{V}}$ such that $g^{n}(V) \supset B(g^{n}(y),r)$ for $n\geq n_V$.
\end{proof}

The proof of the following lemma can be found in \cite[Lemma 2.2.5]{ilp}.

\begin{lema}\label{lema77}
 Given $\varepsilon >0,$ there exist open sets $B_1,...,B_n$ and $m\in\N$ such that:
 \begin{enumerate}
  \item[(i)] $\cup_{i=1}^{n}B_i=\T^{2}$, $diam (B_i)<\varepsilon $ and $A^{m}(B_i)=\T^{2}$
  for every $i=1\ldots,n$;
  \item[(ii)] There exists a $C^{0}$-neighborhood $\mathcal{U}_A$ of $A$ such that
	      $g^{m}(B_i)=\T^{2}$  for all $g\in  \mathcal{U}_A,$ for every $i=1\ldots,n$.
 \end{enumerate}

\end{lema}

\begin{proof}[{\textbf{Proof of the fact that $g$ is transitive.}}]
To prove that $g$ is transitive,
it is enough to prove that given $V\subset \T^{2}$ there exists $m$ such that $g^{m}(V)=\T^{2}$.\\

For $\varepsilon =r/4$,  let $\mathcal{U}_A$ be a $C^{0}$-neighborhood of $A$ such that Lemma
\ref{lema77} holds. We can assume that $\mathcal{U}_h \subset \mathcal{U}_A$ ($\mathcal{U}_h$
has been defined at the beginning of the section) and $g\in \mathcal{U}_h $.
Let us now prove that $g$ is transitive.

 Given an open set
$V\subset\T^{2}$, Lemma \ref{lema3} implies that there exist $y\in V$ and
$n_{_{V}}\in\N $ such that $g^{n}(V)\supset B(g^{n}(y),
r)$ for all $n\geq n_{_{V}}$ . Let  $B_1,...,B_n$ and $m\in \N$ be %the open sets and the positive integer
given by Lemma \ref{lema77}. As $diam (B_i)<r/4$, there exists $i_0$ such
that $g^{n_{_{V}}}(V)\supset B_{i_{0}}$. By Lemma \ref{lema77}, item (ii), $g^{m}(B_{i_{0}})=\T^{2}$
so  $g^{n_{_{V}}+m}(V)=\T^{2}$, which implies that $g$ es transitive.

\end{proof}

\noindent{\it{Acknowledgment}}: We would like to thank Matilde Martinez for helpful discussions and for many useful remarks on
mathematical structure, style, references, etc.
\\

\end{document}